\documentclass[12pt]{article}%
\usepackage{graphicx}
\usepackage[intlimits]{amsmath}
\usepackage{latexsym}
\usepackage{amsfonts}
\usepackage{amssymb}%
\makeatletter
\newfont{\footsc}{cmcsc10 at 8truept}
\newfont{\footbf}{cmbx10 at 8truept}
\newfont{\footrm}{cmr10 at 10truept}
\newtheorem{theorem}{Theorem}

\newtheorem{corollary}[theorem]{Corollary}

\newtheorem{problem}[theorem]{Problem}

\newenvironment{proof}[1][Proof]{\noindent{\textbf {#1}  }}  {\hfill$\Box$\bigskip}

\begin{document}

\title{Maximum norms of graphs and matrices, and their complements}
\author{Vladimir Nikiforov \thanks{Department of Mathematical Sciences, University of
Memphis, Memphis TN 38152, USA; email: \textit{vnikifrv@memphis.edu}}
\thanks{Research supported by NSF grant DMS-0906634.} and Xiying Yuan
\thanks{\textit{Corresponding author. }Department of Mathematics, Shanghai
University, Shanghai, 200444, China; email:
\textit{xiyingyuan2007@hotmail.com}} \thanks{Research supported by National
Science Foundation of China grant No. 11101263.} }
\maketitle

\begin{abstract}
Given a graph $G,$ let $\left\Vert G\right\Vert _{\ast}$ denote the trace norm
of its adjacency matrix, also known as the \textit{energy} of $G.$ The main
result of this paper states that if $G$ is a graph of order $n,$ then
\[
\left\Vert G\right\Vert _{\ast}+\left\Vert \overline{G}\right\Vert _{\ast
}\,\leq\left(  n-1\right)  \left(  1+\sqrt{n}\right)  ,
\]
where $\overline{G}$ is the complement of $G.$ Equality is possible if and
only if $G$ is a strongly regular graph with parameters $\left(  n,\left(
n-1\right)  /2,\left(  n-5\right)  /4,\left(  n-1\right)  /4\right)  ,$ known
also as a conference graph.

In fact, the above problem is stated and solved in a more general setup - for
nonnegative matrices with bounded entries. In particular, this study exhibits
analytical matrix functions attaining maxima on matrices with rigid and
complex combinatorial structure.

In the last section the same questions are studied for Ky Fan norms. Possibe
directions for further research are outlined, as it turns out that the above
problems are just a tip of a larger multidimensional research area.\medskip



\end{abstract}

\section{Introduction and main results}

In this paper we study the maxima of certain norms of nonnegative matrices
with bounded entries. We shall focus first on the trace norm $\left\Vert
A\right\Vert _{\ast}$ of a matrix $A,$ which is just the sum of its singular
values. The trace norm of the adjacency matrix of graphs has been intensively
studied recently under the name \emph{graph energy}. This research has been
initiated by Gutman in \cite{Gut78}; the reader is referred to \cite{LSG12}
for a comprehensive recent survey and references.\ 

One of the most intriguing problems in this area is to determine which graphs
with given number of vertices have maximal energy. A cornerstone result of
Koolen and Moulton \cite{KoMo01} shows that if $G$ is a graph of order $n,$
then the trace norm of its adjacency matrix $\left\Vert G\right\Vert _{\ast}$
satisfies the inequality%
\begin{equation}
\left\Vert G\right\Vert _{\ast}\leq\left(  1+\sqrt{n}\right)  \frac{n}{2},
\label{maxen}%
\end{equation}
with equality holding precisely when $G\ $is a strongly regular graph with
parameters%
\[
\left(  n,\left(  n+\sqrt{n}\right)  /2,\left(  n+2\sqrt{n}\right)  /4,\left(
n+2\sqrt{n}\right)  /4\right)  .
\]
Note that $G$ can be characterized also as a regular graph of degree $\left(
n+\sqrt{n}\right)  /2,$ whose singular values, other than the largest one, are
equal to $\sqrt{n}/2.$ Graphs with such properties can exist only if $n$ is an
even square. It is easy to see that these graphs are related to Hadamard
matrices, and indeed such relations have been outlined in \cite{Hae08}.

As it turns out, if for a graph $G$ equality holds in (\ref{maxen})$,$ then
its complement $\overline{G}$ is a strongly regular graph which is quite
similar but not isomorphic to $G$, and therefore with $\left\Vert \overline
{G}\right\Vert _{\ast}<\left(  1+\sqrt{n}\right)  n/2$. This observation led
Gutman and Zhou \cite{ZhGu07} to the following natural question:\medskip\ 

\emph{What is the maximum }$\mathcal{E}\left(  n\right)  $\emph{ of the sum
}$\left\Vert G\right\Vert _{\ast}+\left\Vert \overline{G}\right\Vert _{\ast}%
,$\emph{ where }$G$\emph{ is a graph of order} $n?$ \medskip

Questions of this type are\ the subject matter of this paper. To begin with,
note that the Paley graphs, and more generally conference graphs (see below),
provide a lower bound%
\begin{equation}
\mathcal{E}\left(  n\right)  \geq\left(  n-1\right)  \left(  1+\sqrt
{n}\right)  , \label{maxen1}%
\end{equation}
which closely complements the upper bound
\[
\mathcal{E}\left(  n\right)  \leq\sqrt{2}n+\left(  n-1\right)  \sqrt{n-1}%
\]
proved by Gutman and Zhou in \cite{ZhGu07}. Although the latter inequality can
be improved by a more careful proof (see Theorem \ref{th3} below), a gap still
remains between the upper and lower bounds on $\mathcal{E}\left(  n\right)  $.
In this note we shall close this gap and show that inequality (\ref{maxen1})
is in fact an equality. This advancement comes at a price of a rather long
proof, involving some new analytical and combinatorial techniques based on
Weyl's inequalities for sums of Hermitian matrices.

Moreover, we prove our bounds for nonnegative matrices that are more general
than adjacency matrices of graphs, and yet the adjacency matrices of
conference graphs provide the cases of equality. It is somewhat surprising
that purely analytical matrix functions attain their maxima on matrices with a
rigid and sophisticated combinatorial structure.

It is also natural to study similar problems for matrix norms different from
the trace norm, like the Ky Fan or the Schatten norms. Other parameters can be
changed independently and we thus arrive at a whole grid of extremal problems
most of which are open and seem rather difficult. Results of this types and
directions for further research are outlined in Section \ref{FD}.\bigskip

For reader's sake we start with our graph-theoretic result first.

\begin{theorem}
\label{mth}If $G$ is a graph of order $n\geq7,$ then
\begin{equation}
\left\Vert G\right\Vert _{\ast}+\left\Vert \overline{G}\right\Vert _{\ast}%
\leq\left(  n-1\right)  \left(  1+\sqrt{n}\right)  , \label{mres}%
\end{equation}
with equality holding if and only if $G$ is a conference graph.
\end{theorem}

Recall that a conference graph of order $n$ is a strongly regular graph with
parameters%
\[
\left(  n,\left(  n-1\right)  /2,\left(  n-5\right)  /4,\left(  n-1\right)
/4\right)  .
\]
It is known that the eigenvalues of a conference graph of order $n$ are
\[
\left(  n-1\right)  /2,\left(  (\sqrt{n}-1)/2\right)  ^{\left[  \left(
n-1\right)  /2\right]  },\left(  -\left(  \sqrt{n}+1\right)  /2\right)
^{\left[  \left(  n-1\right)  /2\right]  },
\]
where the numbers in brackets denote multiplicities. We shall make use of the
fact that every graph with these eigenvalues must be a conference graph.
Reference material on these questions can be found in \cite{GoRo01}.

The best known type of conference graphs are the Paley graphs $P_{q}:$
\emph{Given a prime power }$q=1$\emph{ }$(\operatorname{mod}$\emph{ }%
$4),$\emph{ the vertices of }$P_{q}$\emph{ are the numbers }$1,\ldots,q$\emph{
and two vertices }$u,v$\emph{ are adjacent if }$u-v$\emph{ is an exact square
}$\operatorname{mod}$\emph{ }$q.$ The Paley graphs are self complementary,
which fits with their extremal property with respect to (\ref{mres}).

As usual, $I_{n}$ stands for the identity matrix of size $n,$ $J_{n}$ stands
for the all ones matrix of size $n,$ and $\mathbf{j}_{n}$ is the
$n$-dimensional vector of all ones. Also, $\sigma_{1}\left(  A\right)
\geq\sigma_{2}\left(  A\right)  \geq\cdots$ will denote the singular values of
a matrix $A$ and $\mu_{1}\left(  A\right)  \geq\mu_{2}\left(  A\right)
\geq\cdots$ will denote the eigenvalues of a symmetric matrix $A$. Finally,
$\left\vert A\right\vert _{\infty}$ will stand for the maximum of the absolute
values of the entries of $A.$\medskip

We shall prove Theorem \ref{mth} in the following matrix setup.

\begin{theorem}
\label{th1}If $A$ is a symmetric nonnegative matrix of size $n\geq7,$ with
$\left\vert A\right\vert _{\infty}\leq1,$ and with zero diagonal, then
\begin{equation}
\left\Vert A\right\Vert _{\ast}+\left\Vert J_{n}-I_{n}-A\right\Vert _{\ast
}\leq\left(  n-1\right)  \left(  1+\sqrt{n}\right)  , \label{main1}%
\end{equation}
with equality holding if and only if $A$ is the adjacency matrix of a
conference graph.
\end{theorem}

A crucial role in our proof play the following two facts about nonnegative
matrices, which are of interest on their own.

\begin{theorem}
\label{th2}If $A$ is a square nonnegative matrix of size $n,$ with $\left\vert
A\right\vert _{\infty}\leq1,$ and with zero diagonal, then
\begin{equation}
\left\Vert A+\frac{1}{2}I_{n}\right\Vert _{\ast}+\left\Vert J_{n}-A-\frac
{1}{2}I_{n}\right\Vert _{\ast}\leq n+(n-1)\sqrt{n}. \label{th2in}%
\end{equation}
Equality is possible if and only if $A$ is a $\left(  0,1\right)  $-matrix,
with all row and column sums equal to $\left(  n-1\right)  /2,$ and such that
$\sigma_{i}\left(  A+\frac{1}{2}I_{n}\right)  =\sqrt{n}/2$ for $i=2,\ldots,n.$
\end{theorem}

\begin{corollary}
\label{cor2}If $A$ is a symmetric nonnegative matrix of size $n,$ with
$\left\vert A\right\vert _{\infty}\leq1,$ and with zero diagonal, such that
\begin{equation}
\left\Vert A+\frac{1}{2}I_{n}\right\Vert _{\ast}+\left\Vert J_{n}-A-\frac
{1}{2}I_{n}\right\Vert _{\ast}=n+(n-1)\sqrt{n}, \label{cor2in}%
\end{equation}
then $A$ is the adjacency matrix of a conference graph.
\end{corollary}

Before proceeding with the proof of Theorem \ref{th1} let us note that the
difficulty of its proof seems due to the fact that $A$ is a symmetric matrix
and has a zero diagonal. In Theorem \ref{th3} we shall see that if these
conditions are omitted, the proof becomes really straightforward, but
unfortunately this result is not tight for graphs.

\section{Proofs of Theorems \ref{th2} and \ref{th1}}

Our proof of Theorem \ref{th2} illustrates the two main ingredients of several
proofs later. First, this is an application of the \textquotedblleft
arithmetic mean-quadratic mean\textquotedblright, or the AM-QM inequality.
This way we obtain an upper bound on the trace norm by the sum of the squares
of the singular values of a matrix, which is equal to the sum of the squares
of its entries. Second, we identify a function of the type
\[
f(x)=x+\sqrt{a-bx^{2}}%
\]
and conclude that $f(x)$ is decreasing in $x$ under certain assumptions about
$a,b$ and $x.$ This way we obtain the upper bound $f(x)\leq f\left(  \min
x\right)  .$\bigskip

\begin{proof}
[\textbf{Proof of Theorem \ref{th2}}]Set for short%
\begin{align*}
B  &  =A+\frac{1}{2}I_{n},\\
\overline{B}  &  =J_{n}-A-\frac{1}{2}I_{n}.
\end{align*}
Applying the AM-QM inequality, we see that
\begin{align}
\left\Vert B\right\Vert _{\ast}+\left\Vert \overline{B}\right\Vert _{\ast}  &
=\sigma_{1}\left(  B\right)  +\sigma_{1}\left(  \overline{B}\right)  +%
%TCIMACRO{\dsum \limits_{i=2}^{n}}%
%BeginExpansion
{\displaystyle\sum\limits_{i=2}^{n}}
%EndExpansion
\sigma_{i}\left(  B\right)  +\sigma_{i}\left(  \overline{B}\right)
\label{in1}\\
&  \leq\sigma_{1}\left(  B\right)  +\sigma_{1}\left(  \overline{B}\right)
+\sqrt{2(n-1)\left(
%TCIMACRO{\dsum \limits_{i=2}^{n}}%
%BeginExpansion
{\displaystyle\sum\limits_{i=2}^{n}}
%EndExpansion
\sigma_{i}^{2}\left(  B\right)  +\sigma_{i}^{2}\left(  \overline{B}\right)
\right)  }.\nonumber
\end{align}
On the other hand,
\begin{equation}%
%TCIMACRO{\dsum \limits_{i=1}^{n}}%
%BeginExpansion
{\displaystyle\sum\limits_{i=1}^{n}}
%EndExpansion
\sigma_{i}^{2}(B)+\sigma_{i}^{2}(\overline{B})=tr(BB^{T})+tr(\overline
{B}\text{ }\overline{B}^{T})=\sum_{i,j}\left(  B_{ij}^{2}+\overline{B}%
_{ij}^{2}\right)  \leq n^{2}-\frac{n}{2},\text{ } \label{in2}%
\end{equation}
and so
\begin{align}
\left\Vert B\right\Vert _{\ast}+\left\Vert \overline{B}\right\Vert _{\ast}  &
\leq\sigma_{1}\left(  B\right)  +\sigma_{1}\left(  \overline{B}\right)
+\sqrt{2(n-1)\left(  n^{2}-\frac{n}{2}-\sigma_{1}^{2}(B)-\sigma_{1}%
^{2}(\overline{B})\right)  }\nonumber\\
&  \leq\sigma_{1}\left(  B\right)  +\sigma_{1}\left(  \overline{B}\right)
+\sqrt{2(n-1)\left(  n^{2}-\frac{n}{2}-\frac{(\sigma_{1}\left(  B\right)
+\sigma_{1}\left(  \overline{B}\right)  )^{2}}{2}\right)  }.\nonumber
\end{align}

Furthermore, using calculus we see that the function
\[
f(x)=x+\sqrt{2(n-1)\left(  n^{2}-\frac{n}{2}-\frac{x^{2}}{2}\right)  }%
\]
is decreasing in $x$ when $x\geq n.$ On the other hand, $\sigma_{1}\left(
A\right)  $ is the operator norm of $A;$ hence
\begin{equation}
\sigma_{1}(B)+\sigma_{1}(\overline{B})\geq\frac{1}{n}\left\langle
B\mathbf{j}_{n},\mathbf{j}_{n}\right\rangle +\frac{1}{n}\left\langle
\overline{B}\mathbf{j}_{n},\mathbf{j}_{n}\right\rangle =\frac{1}%
{n}\left\langle J_{n}\mathbf{j}_{n},\mathbf{j}_{n}\right\rangle =n.
\label{opnor}%
\end{equation}
Thus, $f\left(  \sigma_{1}\left(  B\right)  +\sigma_{1}\left(  \overline
{B}\right)  \right)  \leq f\left(  n\right)  ,$ and so
\[
\left\Vert B\right\Vert _{\ast}+\left\Vert \overline{B}\right\Vert _{\ast}\leq
n+\sqrt{2(n-1)\left(  n^{2}-\frac{n}{2}-\frac{n^{2}}{2}\right)  },
\]
completing the proof of (\ref{th2in}).

If equality holds in (\ref{th2in}), then we have equality in (\ref{in1}),
(\ref{in2}), and (\ref{opnor}). Therefore, $A$ is a $\left(  0,1\right)
$-matrix,
\begin{align*}
\sigma_{1}\left(  B\right)  +\sigma_{1}\left(  \overline{B}\right)   &  =n,\\
\sigma_{1}^{2}\left(  B\right)  +\sigma_{1}^{2}\left(  \overline{B}\right)
&  =(\sigma_{1}\left(  B\right)  +\sigma_{1}\left(  \overline{B}\right)
)^{2}/2,
\end{align*}
and
\[
\sigma_{2}\left(  B\right)  =\cdots=\sigma_{n}\left(  B\right)  =\sigma
_{2}\left(  \overline{B}\right)  =\cdots=\sigma_{n}\left(  \overline
{B}\right)  .
\]
Hence $\sigma_{1}\left(  B\right)  =\sigma_{1}\left(  \overline{B}\right)
=n/2$ and%
\[
\sigma_{2}\left(  B\right)  =\cdots=\sigma_{n}\left(  B\right)  =\sqrt{n}/2.
\]
We omit the simple proof that if $\sigma_{1}(B)=\frac{1}{n}\left\langle
B\mathbf{j}_{n},\mathbf{j}_{n}\right\rangle ,$ then all row and column sums of
$B$ are equal, which in our case implies that all row and column sums of $A$
are equal to $\left(  n-1\right)  /2.$

To prove that the fact that if $A$ satisfies the listed conditions, then
equality holds in (\ref{th2in}) it is enough to check that
\[
\sigma_{i}\left(  \overline{B}\right)  =\sqrt{n}/2
\]
for $i=2,\ldots,n$. Indeed, since the row and column sums of $B$ are equal to
$n/2,$ we see that
\[
\overline{B}\text{ }\overline{B}^{T}=\left(  J_{n}-B\right)  \left(
J_{n}-B^{T}\right)  =nJ_{n}-\frac{n}{2}J_{n}-\frac{n}{2}J_{n}+BB^{T}=BB^{T}.
\]
and so, $\sigma_{i}\left(  \overline{B}\right)  =\sigma_{i}\left(  B\right)
=\sqrt{n}/2.$
\end{proof}

\begin{proof}
[Proof of Corollary \ref{cor2}]From Theorem \ref{th2} we know that $A$ is a
symmetric $\left(  0,1\right)  $ matrix of a regular graph of degree $\left(
n-1\right)  /2.$ Since $\mu_{i}\left(  A+\frac{1}{2}I_{n}\right)  =\pm\sqrt
{n}/2$ for $i=2,\ldots,n,$ then $\mu_{i}\left(  A\right)  =\left(  \sqrt
{n}-1\right)  /2$ or $\mu_{i}\left(  A\right)  =-\left(  \sqrt{n}+1\right)
/2$ for $i=2,\ldots,n.$ Using the fact $\mu_{1}\left(  A\right)  =\left(
n-1\right)  /2,$ we see that $A$ has the spectrum of a conference graph, and
therefore $A$ is the adjacency matrix of a conference graph. This completes
the proof of Corollary \ref{cor2}.
\end{proof}

\begin{proof}
[\textbf{Proof of Theorem \ref{th1}}]Our main goal is to prove inequality
(\ref{main1}). To keep the proof streamlined we have freed it of a large
number of easy, but tedious calculations.

Assume that $n\geq7$ and let $A$ be a matrix satisfying the conditions of the
theorem and such that
\[
\left\Vert A\right\Vert _{\ast}+\left\Vert J_{n}-I_{n}-A\right\Vert _{\ast}%
\]
is maximal. For short, let
\[
\overline{A}=J_{n}-I_{n}-A,
\]
and set
\[
\mu_{k}=\mu_{k}\left(  A\right)  ,\text{ \ }\overline{\mu}_{k}=\mu_{k}\left(
\overline{A}\right)
\]
for every $k\in\left[  n\right]  .$

We start with the following particular case of Weyl's inequalities: for every
$k=2,\ldots,n,$
\begin{equation}
\mu_{k}+\overline{\mu}_{n-k+2}\leq\mu_{2}\left(  J_{n}-I_{n}\right)  =-1.
\label{Wein}%
\end{equation}

Write $n^{+}\left(  A\right)  $ for the number of nonnegative eigenvalues of a
matrix $A.$ To keep track of the signs of $\mu_{k}$ and $\overline{\mu
}_{n-k+2},$ define the set
\[
P=\left\{  k\text{ }|\text{ }2\leq k\leq n,\text{ }\mu_{k}\geq0\text{ or
}\overline{\mu}_{n-k+2}\geq0\right\}  ,
\]
and let $p=\left\vert P\right\vert .$ Note that if $k\in P$ and $\mu_{k}%
\geq0,$ then (\ref{Wein}) implies that $\overline{\mu}_{n-k+2}<0.$ Thus, in
view of $\mu_{1}\geq0$ and $\overline{\mu}_{1}\geq0,$ we see that
$n^{+}\left(  A\right)  +n^{+}\left(  \overline{A}\right)  =p+2$.

The pivotal point of our proof is the value of $p,$ which obviously is at most
$n-1.$ If $p=n-1,$ we shall finish the proof by Theorem \ref{th2} and
Corollary \ref{cor2}. In the remaining cases, when $p=n-2$ or $p<n-2,$ we
shall show that strict inequality holds in (\ref{main1}). We note that these
two cases require distinct proofs, albeit very similar in spirit.

Let first $p=n-1.$ Set for short
\[
B=A+\frac{1}{2}I_{n},\text{ \ }\overline{B}=J_{n}-A-\frac{1}{2}I_{n}%
=\overline{A}+\frac{1}{2}I_{n},
\]
and note that
\[%
%TCIMACRO{\dsum \limits_{i=1}^{n}}%
%BeginExpansion
{\displaystyle\sum\limits_{i=1}^{n}}
%EndExpansion
\mu_{i}(B)=tr\left(  B\right)  =\frac{n}{2},
\]
which implies that
\[%
%TCIMACRO{\dsum \limits_{\mu_{i}(B)<0}}%
%BeginExpansion
{\displaystyle\sum\limits_{\mu_{i}(B)<0}}
%EndExpansion
\left\vert \mu_{i}(B)\right\vert =%
%TCIMACRO{\dsum \limits_{\mu_{i}(B)\geq0}}%
%BeginExpansion
{\displaystyle\sum\limits_{\mu_{i}(B)\geq0}}
%EndExpansion
\mu_{i}(B)-\frac{n}{2},
\]
and by symmetry, also that%
\[%
%TCIMACRO{\dsum \limits_{\mu_{i}(\overline{B})<0}}%
%BeginExpansion
{\displaystyle\sum\limits_{\mu_{i}(\overline{B})<0}}
%EndExpansion
\left\vert \mu_{i}(\overline{B})\right\vert =%
%TCIMACRO{\dsum \limits_{\mu_{i}(\overline{B})\geq0}}%
%BeginExpansion
{\displaystyle\sum\limits_{\mu_{i}(\overline{B})\geq0}}
%EndExpansion
\mu_{i}(\overline{B})-\frac{n}{2}.
\]
Hence, we find that
\begin{equation}
\left\Vert B\right\Vert _{\ast}+\left\Vert \overline{B}\right\Vert _{\ast}=%
%TCIMACRO{\dsum \limits_{i=1}^{n}}%
%BeginExpansion
{\displaystyle\sum\limits_{i=1}^{n}}
%EndExpansion
\left\vert \mu_{i}(B)\right\vert +\left\vert \mu_{i}(\overline{B})\right\vert
=2%
%TCIMACRO{\dsum \limits_{\mu_{i}(B)\geq0}}%
%BeginExpansion
{\displaystyle\sum\limits_{\mu_{i}(B)\geq0}}
%EndExpansion
\mu_{i}(B)+2%
%TCIMACRO{\dsum \limits_{\mu_{i}(\overline{B})\geq0}}%
%BeginExpansion
{\displaystyle\sum\limits_{\mu_{i}(\overline{B})\geq0}}
%EndExpansion
\mu_{i}(\overline{B})-n. \label{in3}%
\end{equation}
On the other hand, we see that
\begin{align*}%
%TCIMACRO{\dsum \limits_{\mu_{i}(B)\geq0}}%
%BeginExpansion
{\displaystyle\sum\limits_{\mu_{i}(B)\geq0}}
%EndExpansion
\mu_{i}(B)+%
%TCIMACRO{\dsum \limits_{\mu_{i}(\overline{B})\geq0}}%
%BeginExpansion
{\displaystyle\sum\limits_{\mu_{i}(\overline{B})\geq0}}
%EndExpansion
\mu_{i}(\overline{B})  &  =%
%TCIMACRO{\dsum \limits_{\mu_{i}\geq-1/2}}%
%BeginExpansion
{\displaystyle\sum\limits_{\mu_{i}\geq-1/2}}
%EndExpansion
\left\vert \mu_{i}+1/2\right\vert +%
%TCIMACRO{\dsum \limits_{\overline{\mu}_{i}\geq-1/2}}%
%BeginExpansion
{\displaystyle\sum\limits_{\overline{\mu}_{i}\geq-1/2}}
%EndExpansion
\left\vert \overline{\mu}_{i}+1/2\right\vert \\
&  \geq%
%TCIMACRO{\dsum \limits_{\mu_{i}\geq0}}%
%BeginExpansion
{\displaystyle\sum\limits_{\mu_{i}\geq0}}
%EndExpansion
\left(  \mu_{i}+1/2\right)  +%
%TCIMACRO{\dsum \limits_{\overline{\mu}_{i}\geq0}}%
%BeginExpansion
{\displaystyle\sum\limits_{\overline{\mu}_{i}\geq0}}
%EndExpansion
\left(  \overline{\mu}_{i}+1/2\right) \\
&  =\frac{1}{2}\left\Vert A\right\Vert _{\ast}+\frac{1}{2}\left\Vert
\overline{A}\right\Vert _{\ast}+\frac{1}{2}n^{+}\left(  A\right)  +\frac{1}%
{2}n^{+}\left(  \overline{A}\right) \\
&  =\frac{1}{2}\left\Vert A\right\Vert _{\ast}+\frac{1}{2}\left\Vert
\overline{A}\right\Vert _{\ast}+\frac{1}{2}\left(  p+2\right)
\end{align*}

Therefore, in view of (\ref{in3}),%
\[
\left\Vert B\right\Vert _{\ast}+\left\Vert \overline{B}\right\Vert _{\ast}%
\geq\left\Vert A\right\Vert _{\ast}+\left\Vert \overline{A}\right\Vert _{\ast
}+p+2-n=\left\Vert A\right\Vert _{\ast}+\left\Vert \overline{A}\right\Vert
_{\ast}+1.
\]
and inequality (\ref{main1}) follows by Theorem \ref{th2} applied to the
matrix $B.$ This completes the proof when $p=n-1.$ Note that the
characterization of equality in (\ref{main1}) comes directly from Corollary
\ref{cor2}, as in the cases when $p=n-2$ or $p<n-2,$ a strict inequality
always holds in (\ref{main1}).

Let now $p=n-2.$ That is to say, there exists exactly one $k\in\left\{
2,\ldots,n\right\}  $ such that $\mu_{k}<0$ and $\overline{\mu}_{n-k+2}<0.$
Then, setting
\begin{align*}
x  &  =\mu_{1}+\overline{\mu}_{1},\\
y  &  =\left\vert \mu_{k}\right\vert +\left\vert \overline{\mu}_{n-k+2}%
\right\vert ,
\end{align*}
we see that $y\geq1,$ and also
\begin{equation}
x=\mu_{1}+\overline{\mu}_{1}=-%
%TCIMACRO{\dsum \limits_{i=2}^{n}}%
%BeginExpansion
{\displaystyle\sum\limits_{i=2}^{n}}
%EndExpansion
\mu_{i}+\overline{\mu}_{n-i+2}\geq y+n-2\geq n-1. \label{inx1}%
\end{equation}

By the definition of $P$ and Weyl's inequalities (\ref{Wein}), for each $i\in
P,$ we have%
\[
\mu_{i}^{2}+\overline{\mu}_{n-i+2}^{2}=\frac{\left(  \left\vert \mu
_{i}\right\vert +\left\vert \overline{\mu}_{n-i+2}\right\vert \right)  ^{2}%
}{2}+\frac{\left(  \left\vert \mu_{i}\right\vert -\left\vert \overline{\mu
}_{n-i+2}\right\vert \right)  ^{2}}{2}\geq\frac{\left(  \left\vert \mu
_{i}\right\vert +\left\vert \overline{\mu}_{n-i+2}\right\vert \right)  ^{2}%
}{2}+\frac{1}{2}.
\]
Therefore,
\begin{align*}
n(n-1)  &  \geq\sum_{i,j}A_{ij}^{2}+\overline{A}_{ij}^{2}=\sum_{i=1}^{n}%
\mu_{i}^{2}+\overline{\mu}_{i}^{2}\\
&  =\mu_{1}^{2}+\overline{\mu}_{1}^{2}+\mu_{k}^{2}+\overline{\mu}_{n-k+2}^{2}+%
%TCIMACRO{\dsum \limits_{i\in P}}%
%BeginExpansion
{\displaystyle\sum\limits_{i\in P}}
%EndExpansion
\mu_{i}^{2}+\overline{\mu}_{n-i+2}^{2}\\
&  \geq\frac{x^{2}}{2}+\frac{y^{2}}{2}+%
%TCIMACRO{\dsum \limits_{i\in P}}%
%BeginExpansion
{\displaystyle\sum\limits_{i\in P}}
%EndExpansion
\frac{\left(  \left\vert \mu_{i}\right\vert +\left\vert \overline{\mu}%
_{n-i+2}\right\vert \right)  ^{2}}{2}+\frac{p}{2}\\
&  \geq\frac{x^{2}}{2}+\frac{y^{2}}{2}+\frac{1}{2p}\left(
%TCIMACRO{\dsum \limits_{i\in P}}%
%BeginExpansion
{\displaystyle\sum\limits_{i\in P}}
%EndExpansion
\left\vert \mu_{i}\right\vert +\left\vert \overline{\mu}_{n-i+2}\right\vert
\right)  ^{2}+\frac{p}{2}.
\end{align*}

Replacing $p$ by $n-2,$ after some simple algebra we find that
\[
\left\Vert A\right\Vert _{\ast}+\left\Vert \overline{A}\right\Vert _{\ast}\leq
x+y+\sqrt{2\left(  n-2\right)  \left(  n(n-1)-\frac{n-2}{2}-\frac{x^{2}}%
{2}-\frac{y^{2}}{2}\right)  }.
\]

We shall show that the function
\[
f(x,y)=x+y+\sqrt{2(n-2)\left(  n(n-1)-\frac{n-2}{2}-\frac{x^{2}}{2}%
-\frac{y^{2}}{2}\right)  }%
\]
is decreasing in $x$ for $x\geq n-1$ and $y\geq1.$ Indeed, otherwise there
exist $x\geq n-1$ and $y\geq1$ such that
\[
1-\frac{(n-2)x}{\sqrt{2(n-2)\left(  n(n-1)-\frac{n-2}{2}-\frac{x^{2}}{2}%
-\frac{y^{2}}{2}\right)  }}=\frac{\partial f(x,y)}{\partial x}\geq0,
\]
which is a contradiction for $n\geq5.$

Now, (\ref{inx1}) implies that
\begin{align*}
f(x,y)  &  \leq f(y+n-2,y)\\
&  =2y+n-2+\sqrt{(n-2)\left(  2n(n-1)-(n-2)-(y+n-2)^{2}-y^{2}\right)  }.
\end{align*}

Furthermore, using calculus, we see that $f(y+n-2,y)$ is decreasing in $y$ for
$n\geq6$ and $y\geq1.$ Therefore,
\[
f(y+n-2,y)\leq f(n-1,1)=n+\sqrt{n(n-1)\left(  n-2\right)  },
\]
and so
\[
\left\Vert A\right\Vert _{\ast}+\left\Vert \overline{A}\right\Vert _{\ast}\leq
n+\sqrt{n(n-1)\left(  n-2\right)  }.
\]
It is not hard to see that for $n\geq6$,
\[
n+\sqrt{n(n-1)\left(  n-2\right)  }<n-1+(n-1)\sqrt{n},
\]
completing the proof of (\ref{main1}) when $p=n-2.$

Let now $p\leq n-3.$ Then there exist two distinct $k,j\in\left\{
2,\ldots,n\right\}  \backslash P.$ Let%
\begin{align*}
x  &  =\mu_{1}+\overline{\mu}_{1},\\
y  &  =\left\vert \mu_{k}\right\vert +\left\vert \overline{\mu}_{n-k+2}%
\right\vert +\left\vert \mu_{j}\right\vert +\left\vert \overline{\mu}%
_{n-j+2}\right\vert
\end{align*}
From the definition of $P$ and Weyl's inequalities (\ref{Wein}) we have
\[
y=\left\vert \mu_{k}\right\vert +\left\vert \overline{\mu}_{n-k+2}\right\vert
+\left\vert \mu_{j}\right\vert +\left\vert \overline{\mu}_{n-j+2}\right\vert
\geq2,
\]
and also, by $tr\left(  A\right)  =tr\left(  \overline{A}\right)  =0,$
\begin{equation}
x=\mu_{1}+\overline{\mu}_{1}=-%
%TCIMACRO{\dsum \limits_{i=2}^{n}}%
%BeginExpansion
{\displaystyle\sum\limits_{i=2}^{n}}
%EndExpansion
\mu_{i}+\overline{\mu}_{n-i+2}\geq y+n-3\geq n-1. \label{inx2}%
\end{equation}

As in the previous case, we obtain
\[
\left\Vert A\right\Vert _{\ast}+\left\Vert \overline{A}\right\Vert _{\ast}\leq
x+y+\sqrt{2(n-3)\left(  n(n-1)-\frac{1}{2}-\frac{x^{2}}{2}-\frac{y^{2}}%
{4}\right)  }.
\]

Next, using calculus, we find that the function
\[
f\left(  x,y\right)  =x+y+\sqrt{2(n-3)\left(  n(n-1)-\frac{1}{2}-\frac{x^{2}%
}{2}-\frac{y^{2}}{4}\right)  }%
\]
is decreasing in $x$ for $n\geq5,$ $x\geq n-1$ and $y\geq2$. Thus,
(\ref{inx2}) implies that
\[
f(x,y)\leq f(y+n-3,y).
\]
Again using calculus, we find that for $n\geq7$ and $y\geq2,$ the function
\[
f(y+n-3,y)=2y+n-3+\sqrt{2(n-3)\left(  n(n-1)-\frac{1}{2}-\frac{(y+n-3)^{2}}%
{2}-\frac{y^{2}}{4}\right)  }%
\]
is decreasing in $y$. Therefore,
\begin{align*}
\left\Vert A\right\Vert _{\ast}+\left\Vert \overline{A}\right\Vert _{\ast}  &
\leq f\left(  x,y\right)  \leq f(y+n-3,y)\leq f(n-1,2)\\
&  =n+1+\sqrt{(n-3)\left(  n^{2}-4\right)  }.
\end{align*}
A simple calculation shows that for $n\geq5$
\[
n+1+\sqrt{(n-3)\left(  n^{2}-4\right)  }<n-1+(n-1)\sqrt{n},
\]
completing the proof of Theorem \ref{th1}.
\end{proof}

\section{\label{FD}Further extensions}

An obvious question that arises from the above results is the possibility to
extend them to non-symmetric and possibly non-square, nonnegative matrices. We
state such extensions in Theorem \ref{th3} and \ref{th3a} below.

Also, the above results focus on the trace norm of matrices, which is known to
be a particular case of the Ky Fan $k$-norms on the one hand, and of the
Schatten $p$-norms on the other. Extremal properties of these norms have been
studied in connection to graphs and nonnegative matrices in \cite{Nik11} and
\cite{Nik12}. It has been shown that many known results for graph energy carry
over these more general norms. Below, the same trend is illustrated for the
extremal problems discussed above.

We denote the Ky Fan $k$-norm of a matrix $A$ by $\left\Vert A\right\Vert
_{\ast k},$ that is to say%
\[
\left\Vert A\right\Vert _{\ast k}=\sigma_{1}\left(  A\right)  +\cdots
+\sigma_{k}\left(  A\right)  .
\]
Also, $J_{m,n}$ will stand for the $m\times n$ matrix of all ones. In what
follows we shall focus only on the Ky Fan norms. The Schatten $p$-norms seem
of a slightly different flavor and we shall leave them for further study.

\begin{theorem}
\label{th3}If $2\leq k\leq m\leq n$ and $A$ is a nonnegative matrix of size
$m\times n,$ with $\left\vert A\right\vert _{\infty}\leq1,$ then
\begin{equation}
\left\Vert A\right\Vert _{\ast k}+\left\Vert J_{m,n}-A\right\Vert _{\ast
k}\leq\sqrt{mn}\left(  1+\sqrt{k-1}\right)  . \label{th3in}%
\end{equation}

\end{theorem}

\begin{proof}
Set for short $\overline{A}=J_{m,n}-A.$ Following familiar arguments, we see
that
\begin{align*}
\left\Vert A\right\Vert _{\ast k}+\left\Vert \overline{A}\right\Vert _{\ast
k}  &  =\sigma_{1}\left(  A\right)  +\sigma_{1}\left(  \overline{A}\right)  +%
%TCIMACRO{\dsum \limits_{i=2}^{k}}%
%BeginExpansion
{\displaystyle\sum\limits_{i=2}^{k}}
%EndExpansion
\sigma_{i}\left(  A\right)  +\sigma_{i}\left(  \overline{A}\right) \\
&  \leq\sigma_{1}\left(  A\right)  +\sigma_{1}\left(  \overline{A}\right)
+\sqrt{2\left(  k-1\right)  \left(
%TCIMACRO{\dsum \limits_{i=2}^{k}}%
%BeginExpansion
{\displaystyle\sum\limits_{i=2}^{k}}
%EndExpansion
\sigma_{i}^{2}\left(  A\right)  +\sigma_{i}^{2}\left(  \overline{A}\right)
\right)  }\\
&  \leq\sigma_{1}\left(  A\right)  +\sigma_{1}\left(  \overline{A}\right)
+\sqrt{2\left(  k-1\right)  \left(  mn-\frac{(\sigma_{1}\left(  A\right)
+\sigma_{1}\left(  \overline{A}\right)  )^{2}}{2}\right)  .}%
\end{align*}
Since the function
\[
f(x)=x+\sqrt{2(k-1)\left(  mn-\frac{x^{2}}{2}\right)  }%
\]
is decreasing in $x$ for $x\geq\sqrt{mn}$, and also
\[
\sigma_{1}(A)+\sigma_{1}(\overline{A})\geq\frac{1}{\sqrt{mn}}\left\langle
A\mathbf{j}_{n},\mathbf{j}_{m}\right\rangle +\frac{1}{\sqrt{mn}}\left\langle
\overline{A}\mathbf{j}_{n},\mathbf{j}_{m}\right\rangle =\frac{1}{\sqrt{mn}%
}\left\langle J_{m,n}\mathbf{j}_{n},\mathbf{j}_{m}\right\rangle =\sqrt{mn},
\]
we see that
\[
\left\Vert A\right\Vert _{\ast}+\left\Vert \overline{A}\right\Vert _{\ast}%
\leq\sqrt{mn}+\sqrt{\left(  k-1\right)  mn}=\sqrt{mn}\left(  1+\sqrt
{k-1}\right)  ,
\]
completing the proof of Theorem \ref{th3}.
\end{proof}

We cannot describe exhaustively the cases of equality in (\ref{th3in}).
However we shall describe a general construction proving that (\ref{th3in}) is
exact in a rich set of cases.

\begin{theorem}
\label{th3a}Let $k\geq2$ be an integer for which there is a Hadamard matrix of
size $k-1.$ Let $p,q\geq1$ be arbitrary integers and set $m=2p\left(
k-1\right)  $ and $n=2q\left(  k-1\right)  .$ There exists a $\left(
0,1\right)  $-matrix $A$ of size $m\times n$ such that
\[
\left\Vert A\right\Vert _{\ast k}+\left\Vert J_{m,n}-A\right\Vert _{\ast
k}=\sqrt{mn}\left(  1+\sqrt{k-1}\right)  .
\]

\end{theorem}

\begin{proof}
Indeed let $H$ be a Hadamard matrix of size $k-1,$ and set
\[
H^{\prime}=\left(
\begin{array}
[c]{cc}%
H & -H\\
-H & H
\end{array}
\right)  =\left(
\begin{array}
[c]{cc}%
1 & -1\\
-1 & 1
\end{array}
\right)  \otimes H,
\]
and
\[
A=\frac{1}{2}\left(  \left(  H^{\prime}\otimes J_{p,q}\right)  +J_{m,n}%
\right)  ,
\]
where $\otimes$ denotes the Kronecker product of matrices. First note that $A$
is a $\left(  0,1\right)  $-matrix of size $m\times n.$

Our main goal is to show that $\sigma_{1}\left(  A\right)  =\sqrt{mn}/2$ and
that%
\[
\sigma_{2}\left(  A\right)  =\cdots=\sigma_{k}\left(  A\right)  =\frac
{\sqrt{mn}}{2\sqrt{k-1}}.
\]
It is known that $H$ has $k-1$ singular values which are equal to $\sqrt
{k-1}.$ Therefore $H^{\prime}$ has $k-1$ singular values all equal to
$2\sqrt{k-1}.$ Next, we see that $H^{\prime}\otimes J_{p,q}$ has $k-1$ nonzero
singular values, all equal to $2\sqrt{pq\left(  k-1\right)  }.$ We also see
that the row and column sums of $H^{\prime}$ are $0$ and thus $0$ is a
singular value of $H^{\prime}$ with singular vectors $\mathbf{j}_{2\left(
k-1\right)  },\mathbf{j}_{2\left(  k-1\right)  };$ hence, $0$ is a singular
value of $H^{\prime}\otimes J_{p,q}$ with singular vectors $\mathbf{j}_{m}$
and $\mathbf{j}_{n}.$ Since $\mathbf{j}_{m}$ and $\mathbf{j}_{n}$ are also
singular vectors to the unique nonzero singular value of $J_{m,n},$ it is easy
to see that the singular values of $A$ are exactly as described above.

Hence
\[
\left\Vert A\right\Vert _{\ast k}=\frac{\sqrt{mn}}{2}+\frac{\sqrt{mn}}%
{2\sqrt{k-1}}\left(  k-1\right)  =\frac{\sqrt{mn}}{2}\left(  1+\sqrt
{k-1}\right)  .
\]
On the other hand%
\[
J_{m,n}-A=\frac{1}{2}\left(  \left(  -H^{\prime}\otimes J_{p,q}\right)
+J_{m,n}\right)
\]
and so the $J_{m,n}-A$ has the same singular values as $A,$ since $-H$ is a
Hadamard matrix as well. This complete the proof of Theorem \ref{th3a}.
\end{proof}

Two open problems arise in connection to Theorems \ref{th1}, \ref{th3} and
\ref{th3a}.

\begin{problem}
Describe all matrices $A$ for which equality holds in (\ref{th3in}).
\end{problem}

\begin{problem}
Let $A$ be a symmetric nonnegative matrix of size $n,$ with $\left\vert
A\right\vert _{\infty}\leq1,$ with zero diagonal. Find the maximum of
\[
\left\Vert A\right\Vert _{\ast k}+\left\Vert J_{n}-I_{n}-A\right\Vert _{\ast
k}.
\]

\end{problem}

It seems very likely that the latter problem can be solved along the lines of
Theorem \ref{th2}.

Note that we have stated and proved Theorem \ref{th3} for $k\geq2.$ As it
turns out the important case $k=1$ (operator norm) is indeed particular.

\begin{theorem}
\label{th4}If $A$\emph{ is an }$m\times n$\emph{ nonnegative matrix, with
}$\left\vert A\right\vert _{\infty}\leq1$\emph{, then,}
\begin{equation}
\sigma_{1}\left(  A\right)  +\sigma_{1}\left(  J_{m,n}-A\right)  \leq
\sqrt{2mn}, \label{sig1in}%
\end{equation}
with equality holding if and only if $mn$ is even, and $A$ is a $\left(
0,1\right)  $-matrix with precisely $mn/2$ ones that are contained either in
$n/2$ columns or in $m/2$ rows.
\end{theorem}

\begin{proof}
We have
\begin{equation}
\sigma_{1}^{2}\left(  A\right)  +\sigma_{1}^{2}\left(  J_{m,n}-A\right)
\leq\sum_{i=1}^{m}\sigma_{i}^{2}\left(  A\right)  +\sigma_{i}^{2}\left(
J_{m,n}-A\right)  =\sum_{ij}A_{ij}^{2}+\left(  1-A_{ij}\right)  ^{2}\leq mn,
\label{in10}%
\end{equation}
and inequality (\ref{sig1in}) follows by the the AM-QM inequality.

If equality holds in (\ref{sig1in}), then we have equalities throughout
(\ref{in10}). Therefore, $A$ is a $\left(  0,1\right)  $-matrix of rank $1$,
and has precisely $mn/2$ ones. Thus, the ones of $A$ form a submatrix $B$ of
$A,$ Since $J_{m,n}-A$ is also a rank $1$ matrix, its ones are also contained
in a submatrix $B^{\prime}$ of $J_{m,n}-A$. Then $B$ and $B^{\prime}$ are
submatrices of $J_{m,n\text{ }}$ that do not share entries and together
contain all entries of $J_{m,n\text{ }}$. Clearly $B$ must be of size either
$m/2\times n$ or $m\times n/2.$ This completes the proof of Theorem \ref{th4}.
\end{proof}

In the spirit of the previous problems, one can ask the following question:
\emph{Let }$A$\emph{ be a symmetric nonnegative matrix of size }$n,$\emph{
with }$\left\vert A\right\vert _{\infty}\leq1,$\emph{ and with zero diagonal.
What is the maximum of }%
\[
\mu_{1}\left(  A\right)  +\mu_{1}\left(  J_{n}-I_{n}-A\right)  .
\]
This question turns to be known and difficult, but it has been recently
answered in \cite{Csi09} and \cite{Ter11}.\bigskip

\textbf{Acknowledgement }\ This work was done while the second author was
vistiting the University of Memphis.

\end{document}